\documentclass{amsart}
\usepackage{amssymb}
\usepackage{tikz}

\newcommand*{\DashedArrow}[1][]{\mathbin{\tikz [baseline=-0.25ex,-latex, dashed,#1] \draw [#1] (0pt,
0.5ex) -- (1.3em,0.5ex);}}

\usepackage{amscd}
\usepackage{empheq,amsmath}
\usepackage{graphicx}
\usepackage[cmtip,all]{xy}

\newtheorem{theorem}{Theorem}[section]
\newtheorem{lemma}[theorem]{Lemma}

\theoremstyle{definition}

\newtheorem{definition}[theorem]{Definition}

\theoremstyle{proposition}
\newtheorem{proposition}[theorem]{Proposition}

\newtheorem{remark}[theorem]{Remark}

\numberwithin{equation}{section}

\theoremstyle{main}
\newtheorem{main}{Theorem}

\begin{document}

\title[2-dimensional projective foliations]{ Irreducible components of the space of foliations by surfaces}


\subjclass{37F75 (primary); 32G34, 32S65 (secondary)}

\author[W. Costa e Silva]{W. Costa e Silva}
\address{IRMAR, Universit\'e Rennes 1, Campus de Beaulieu, 35042 Rennes Cedex France}
\curraddr{IRMAR, Universit\'e Rennes 1, Campus de Beaulieu, 35042 Rennes Cedex France}
\email{wancossil@gmail.com}

\date{}




\begin{abstract} Let $\mathcal{F}$ be written as $ f^{*}\left(\mathcal{G}\right)$, where $\mathcal{G}$ is a $1$-dimensional foliation on $ {\mathbb P^{n-1}}$ and $f:{\mathbb P^n}\DashedArrow[->,densely dashed]{\mathbb P^{n-1}}$ a non-linear generic rational map.  We use local stability results of singular holomorphic foliations, to prove that: if $n\geq 4$, a foliation $\mathcal{F}$ by complex surfaces on  $\mathbb P^n$ is globally stable under holomorphic deformations. As a consequence, we obtain irreducible components for the space of two-dimensional foliations in  $\mathbb P^n$. We present also a result which characterizes holomorphic foliations on ${\mathbb P^n}, n\geq 4$ which can be obtained as a pull back of 1- foliations in ${\mathbb P^{n-1}}$ of degree $d\geq2$. \end{abstract}
\maketitle


\setcounter{tocdepth}{1}
\tableofcontents \sloppy

\section{Introduction}
A two  singular foliation $\mathcal F$ of a holomorphic manifold $M$, $dim_{\mathbb C}\geq3$, may be defined by:

\begin{enumerate}

\item A covering $\mathcal U = (U_{\alpha})_{{\alpha \in A}}$ of $M$ by open sets.

\item A collection $(\eta_{\alpha})_{{\alpha \in A}}$ of integrable $(n-2)$-forms, $\eta _{\alpha} \in \Omega^{n-2}(U_{\alpha})$, where $\eta_{\alpha}\not\equiv0$ and defines a $2$-dimensional foliation in $U_{\alpha}$.

\item A multiplicative cocycle $G:=(g_{\alpha\beta})_{U_{\alpha}\cap U_{\beta}\neq0}$ such that $\eta_{\alpha}=g_{\alpha\beta}\eta_{\beta}$.
\end{enumerate}
If $N_{\mathcal F}$ denotes the holomorphic line bundle represented by the cocycle $G$, the family $(\eta_{\alpha})_{{\alpha \in A}}$, defines a holomorphic section of the vector bundle $\Omega^{n-2}(M)\otimes N_{\mathcal F}$ i.e. an element $\eta$ of the cohomology vector space $H^0(M,\Omega^{n-2}(M)\otimes N_{\mathcal F})$. The analytic subset $Sing(\eta):=\{p\in {M}|\eta(p)=0\}$ is the singular set of ${\mathcal F}$. In the case of $M=\mathbb P^{n}$, the $n$-dimensional complex projective space, we have a theorem of Chow-type. Denote by $\pi:\mathbb C^{n+1}\backslash \{0\}\to \mathbb P^{n}$ the natural projection, and consider $\pi^\ast\mathcal F$ of the foliation $\mathcal F$ by $\pi$; with the previous notations, $\pi^\ast\mathcal F$ is defined by $(n-2)$-forms, $\pi^\ast\eta _{\alpha} \in \Omega^{n-2}[\pi^{-1}(U_{\alpha})]$. Recall that for $n\geq2$ we have $H^{1}(\mathbb C^{n+1}\backslash \{0\},\mathcal O^{\ast})=\{1\}$: it is a result from Cartan. As a consequence, there exists a global holomorphic $(n-2)$-form $\eta$ on $\mathbb C^{n+1}\backslash \{0\}$ which defines $\pi^\ast\mathcal F$ on $\mathbb C^{n+1}\backslash \{0\}$. By Hartog's extension theorem, $\eta$ can be extended holomorphicaly at $0$. By construction we have $i_R\eta=0$, where $R$ is the radial vector field. This fact and the integrability condition imply that each coefficient of $\eta$ is a homogeneous polynomial of degree $deg(\mathcal F)+1$. Moreover, if we take  a section by a generic immersion of hyperplane $H:=(i:\mathbb P^{n-1}\to \mathbb P^{n})$, this procedure gives a foliation by curves $i^{\ast}(\mathcal F)$ on $\mathbb P^{n-1}$. We then define the degree of  $\mathcal F$, for short $deg(\mathcal F)$, as the degree of a generic section as before. From now on we will always assume that the singular set of $\mathcal F$ has codimension greater or equal than two. The projectvisation of the set of $n-2$-forms which satisfy the previous conditions will be denoted by $\mathbb{F}{\rm{ol}}\left(d;2,n\right)$, the space of $2-$dimensional foliations on $\mathbb P^{n}$ of degree $d$. Note that $\mathbb{F}{\rm{ol}}\left(d;2,n\right)$ can be considered as a quasi projective algebraic subset of $\mathbb PH^{0}(\mathbb P^{n},\Omega^{n-2}(\mathbb P^{n})\otimes \mathcal O_{\mathbb P^n}(d+n-1))$. In this scenario we have the following:


\vskip0.2cm
\noindent \textbf{{\underline{Problem}: \emph{Describe and classify the irreducible components of \break $\mathbb{F}{\rm{ol}}\left(d;2,n\right)$ on ${\mathbb P^n}$, such that $n\geq 3.$}}}
\vskip0.2cm
 We observe that the classification of the irreducible components of $\mathbb{F}{\rm{ol}}\left(0;2,n\right)$ is given in \cite[Th. 3.8 p. 46]{cede} and that the classification of the irreducible components of $\mathbb{F}{\rm{ol}}\left(1;2,n\right)$ is given in \cite[Th. 6.2 and Cor. 6.3 p. 935-936]{lpt}. 
 We refer the reader to \cite{cede} and \cite{lpt} and references therein for a detailed description of them. 
In the case of foliations of codimension $1$, the definitions of foliation and degree are analogous and we denote by $\mathbb{F}{\rm{ol}}\left(k,n\right)$ the space of codimension $1$ foliations of degree $k$ on ${\mathbb P^n}$, such that $n\geq 3$. The study of irreducible components of these spaces has been initiated by Jouanolou in \cite{jou}, where  the irreducible components of  $\mathbb{F}{\rm{ol}}\left(k,n\right)$ for $k=0$ and $k=1$ are described. In the case of codimension one foliations  one can exhibit some kind of list of irreducible components in every degree, but this list is incomplete.  In the  paper \cite{cln}, the authors proved that $\mathbb{F}{\rm{ol}}\left(2,n\right)$ has six irreducible components, which can be described by geometric and dynamic properties of a generic element. We refer the reader to \cite{cln} and \cite{ln}  for a detailed description of them.
 There are known families of irreducible components in which the typical element is a pull-back of a foliation on ${\mathbb P^2}$ by a rational map. Given a generic rational map $f: {\mathbb P^n}  \DashedArrow[->,densely dashed    ]   {\mathbb P^2}$ of degree $\nu\geq1$, it can be written in homogeneous coordinates as $f=(F_0,F_1,F_2)$ where $F_0,F_1$ and $F_2$ are homogeneous polynomials of degree $\nu$. Now consider a foliation $\mathcal G$ on ${\mathbb P^2}$ of degree $d\geq2.$ We can associate to the pair $(f,\mathcal G)$ the pull-back foliation $\mathcal F=f^{\ast}\mathcal G$.
The degree of the foliation $\mathcal F$ is $\nu(d+2)-2$ as proved in \cite{clne}.
Denote by $PB(d,\nu;n)$ the closure in $\mathbb{F}{\rm{ol}}\left(\nu(d+2)-2,n\right)$, $n \geq 3$ of the set of foliations $\mathcal F$ of the form $f^{\ast}\mathcal G$. Since $(f,\mathcal G)\to f^{\ast}\mathcal G$ is an algebraic parametrization of $PB(d,\nu;n)$ it follows that $PB(d,\nu;n)$ is an unirational irreducible algebraic subset of $\mathbb{F}{\rm{ol}}\left(\nu(d+2)-2,n\right)$, $n \geq 3$. We have the following result:
\begin{theorem}\label{teo1.1}      
  $PB(d,\nu;n)$ is a unirational irreducible component of\break $\mathbb{F}{\rm{ol}}\left(\nu(d+2)-2,n\right);$ $n \geq 3$, $\nu\geq1$  and $d \geq 2$.
\end{theorem}
The case $\nu=1$, of linear pull-backs, was proven in \cite{caln}, whereas the case $\nu>1$, of nonlinear pull-backs, was proved in \cite{clne}. The search for new components of pull-back type for the space of codimension 1 foliations was considered in the Ph.D thesis of the author \cite{cs} and after in \cite{cs1}. There we investigated branched rational maps and foliations with algebraic invariant sets of positive dimensions. 
 
Recently, A.Lins Neto in \cite{ln3} generalized the results contained in \cite{ln0} about  singularities of integrable $1$-forms for the $2$-dimensional case and he has obtained as a corollary components of linear pull-back type for the case of $2$-dimensional foliations on ${\mathbb P^n}$. In the present work we will explore the result contained in \cite{ln3} and some ideas contained in \cite{clne} to show that, in fact, there exist families of irreducible components of non-linear pull-back type for the $2$-dimensional case. We would like to mention that in \cite{cupe} the authors have shown that linear pull-back components exist in all codimension. However, the techniques that they use to prove this fact can not be applied to the non-linear case. 
 
 \subsection{The present work}
Let us describe, briefly, the type of pull-back foliation that we shall consider.\\ 
 Let us fix some homogeneous coordinates $Z=(z_0,...,z_{n})$ on $\mathbb C^{n+1}$ and $X=(x_0,...,x_{n-1})$ on $\mathbb C^{n}$. Let  $f: {\mathbb P^n}  \DashedArrow[->,densely dashed    ]   {\mathbb P^{n-1}}$ be a rational map represented in the coordinates $Z\in \mathbb{C}^{n+1}$ and $X \in \mathbb C^{n}$ by $\tilde f=(F_{0},F_{1},...,F_{n-1})$ where $F_{i} \in \mathbb C[X]$ are homogeneous polynomials without common factors of degree $\nu$. 
 Let $\mathcal G$ be a foliation by curves on $\mathbb P^{n-1}$. This foliation can be represented in these coordinates by a homogenous polynomial $(n-2)$-form of the type  $$\Omega=(-1)^{i+k+1}\sum_{i,k} x_{k}P_{i}dx_0\wedge...\wedge \widehat{dx_i}\wedge...\wedge\widehat{dx_k}\wedge...\wedge dx_{n-1}$$ for all $i,k \in \{0,...,n-1\}$ where each $P_i$ is a homogeneous polynomial of degree $d$. The pull back foliation $f^{*}(\mathcal G)$ is then defined in homogeneous coordinates by the $(n-2)$-form $$\tilde\eta_{[f,\mathcal G]}\left(Z\right)=\left[
 (-1)^{i+k+1}\sum_{i,k} F_{k}(P_{i}\circ\tilde{f})\space dF_0\wedge...\wedge \widehat{dF_i}\wedge...\wedge\widehat{dF_k}\wedge...\wedge dF_{n-1}\right],$$ $i,k \in \{0,...,n-1\}$ where each coefficient of  $\tilde\eta_{[f,\mathcal G]}\left(W\right)$ has degree \break $\Theta(\nu,d,n)+1=[(d+n-1)\nu-2].$ 
Let $PB(\Theta,2,n)$ be the closure in $\mathbb{F}{\rm{ol}}\left(\Theta;2,n\right)$  of the set $\left\{ \left[\tilde\eta_{[f,\mathcal G]} \right] \right\}$, where $\tilde\eta_{[f,\mathcal G]}$ is as before. The pull-back foliation's degree is\break $\Theta(\nu,d,n)=[(d+n-1)\nu-3]$ and for simplicity we will denote it by $\Theta$. Let us state the main result of this work.
\begin{main}\label{teob}$PB(\Theta;2,n)$ is a unirational irreducible component of $\mathbb{F}{\rm{ol}}\left(\Theta;2,n\right)$ for all $n \geq 4$, $\nu\geq2$ and $d \geq 2$. 
\end{main}
It is worth pointing out that the case $n=3$ is also true and it is contained in theorem \ref{teo1.1}. So we can think this result as the $n\geq4$-dimensional generalization of \cite{clne} for $2$-dimensioal foliations in $\mathbb P^{n}$

\subsubsection*{Acknowledgments:} I am deeply grateful to D. Cerveau for the discussions, suggestions and comments. I would also like to thank E. Goulart for the corrections in the manuscript. This work was developed at IRMAR(Rennes, France) and was supported by Capes-Brasil, process number (9814-13-2).

\section{$1$-dimensional foliations on $\mathbb P^{n-1}$}

\subsection{Basic facts} 
We recall the basic definitions and properties of foliations by curves on $\mathbb P^{n-1}$ that we will use in this work. Proofs and details can be found in \cite{lnsoares} and \cite{ln3}. \\
\noindent Let  $R=\sum_{i=0}^{n-1}x_i\frac{\partial}{\partial x_i}$ be the radial vetor field in $\mathbb C^n$. Denote by $\Sigma(R,d-1)=\{\mathcal Z|[R,\mathcal Z]=(d-1)\mathcal Z\}$, where $[R,\mathcal Z]$ stands for the Lie's bracket  between the two vector fields $R$ and $\mathcal Z$. Observe that  $\Sigma(R,d-1)$ is a finite dimensional vector space whose elements are homogeneous polynomials of degree $d$. Let us write $\mathcal X=(\mathcal X_0,\dots,\mathcal X_n)$ and $\nabla\mathcal X=\sum_{i=0}^{n-1}\frac{\partial\mathcal X_i}{\partial x_i}$. Let $\mathcal E(R,d-1)=\{\mathcal X \in\Sigma(R,d-1)|\nabla\mathcal X=0\}$, and $\mathcal K(R,d-1)=\{\mathcal X\in\mathcal E (R,d-1)|\mathcal X$ has an isolated singularity at $0\in \mathbb C^n\}$. It can be verified that $\mathcal K(R,d-1)$ is a Zariski open and dense subset of $\Sigma(R,d-1)$ and for each $\mathcal X\in$ $\mathcal K(R,d-1)$ then the $(n-2)$-form $$\Omega= i_Ri_{\mathcal X}d\sigma=(-1)^{i+k+1}\sum_{i,k} x_{k}P_{i}dx_0\wedge...\wedge \widehat{dx_i}\wedge...\wedge\widehat{dx_k}\wedge...\wedge dx_{n-1},$$ where $d\sigma=dx_0\wedge...\wedge {dx_i}\wedge...\wedge{dx_k}\wedge...\wedge dx_{n-1}$ is the volume form in $\mathbb C^n$, $\mathcal X=\sum_{i}P_i\frac{\partial}{\partial x_i}$ and $0 \in Sing(\Omega)$ is a n.g.K singularity, with rotational $(d+n)\mathcal X$ (see section \ref{section5.2}) and \cite{ln3} for more details.  Observe that if cod $Sing(\Omega)\geq2$ then $\Omega$ defines  a 1-dimensional foliation $\mathcal G$ on $\mathbb P^{n-1}$ of degree $d$.

\begin{definition} Let us denote by $\mathbb{F}{\rm{ol}}\left(d;1,n-1\right)$ the set of $1$-dimensional foliations on $\mathbb P^{n-1}$.  
\end{definition}
\begin{theorem} \label{teo2,1}\cite{lnsoares} Given, $n\geq3$, and $d \geq2$ there exists a Zariski open subset $\mathcal M(d)$ of $\mathbb{F}{\rm{ol}}\left(d;1,n-1\right)$ such that any $\mathcal G$ satisfies:
\begin{enumerate}
\item $\mathcal G$ has exactly $N=\frac{d^{n}-1}{d-1}$ hyperbolic singularities and is regular on the complement.
\item $\mathcal G$ has no invariant algebraic curve.
\end{enumerate}
\end{theorem}
Let $X$ be a germ of vector field at $0\in \mathbb C^{n-1}$ with an isolated singularity at $0$ and denote by $\lambda_1,...,\lambda_{n-1}\in \mathbb C$ the spectrum of its linear part. We say that $X$ is hyperbolic at $0$ if none of the quotients $\frac{\lambda_l}{\lambda_l}$ are real.  
We have the following proposition:

\begin{proposition} Let $Q$ be a germ of vector field with a hyperbolic singularity at $0\in \mathbb C^{n-1}$ and denote by $\lambda_1,...,\lambda_{n-1}\in \mathbb C$ its spectrum. Then, there are exactly $n-1$ germs of irreducible invariant analytic invariant curves $\Gamma_1,\Gamma_2,...,\Gamma_{n-1}$ at $0$ where each $\Gamma_l$ is smooth and tangent to the eigendirection corresponding to $\lambda_l$.
\end{proposition}
\noindent In a local coordinate system near the singularity for instance, $0 \in (\mathbb C^{n-1},u)$ where  $u=(u_1,...,u_{n-1})$ the foliation can be written as
$$Q(u)= (\lambda_1u_1)\frac{\partial}{\partial u_1} + \cdots + (\lambda_{n-1}u_{n-1})\frac{\partial}{\partial u_{n-1}}+h.o.t,$$ where $h.o.t$ stands for higher order terms.
Let us denote by $\mathbb P\mathcal K(R,d-1)=L(d)$ and let $A(d):=\mathcal M(d)\cap L(d)$ be their intersection. An element of the open and dense subset $A(d)\subset \mathbb{F}{\rm{ol}}\left(d;1,n-1\right)$ is the well-known generalized Joanoulou's example, see \cite{lnsoares} and \cite{ln3}.

\section {Rational maps}
Let $f : {\mathbb P^n}  \DashedArrow[->,densely dashed    ]   {\mathbb P^{n-1}}$ be a rational map and $\tilde{f}: {\mathbb C^{n+1}} \to {\mathbb C^{n}}$ its natural lifting in homogeneous coordinates. We are considering the same homogeneous coordinates used in the introduction. 

The \emph{indeterminacy locus} of $f$ is, by definition, the set $I\left(f\right)=\Pi_{n}\left(\tilde{f}^{-1}\left(0\right)\right)$. Observe that the restriction $f|_{\mathbb P^n \backslash I\left(f\right)}$ is holomorphic. We characterize the set of rational maps used throughout this text as follows:
\begin{definition} We denote by $RM\left(n,n-1,\nu\right)$ the set of maps $\left\{f: \mathbb P^n  \DashedArrow[->,densely dashed    ]   \mathbb P^{n-1}\right\}$ of degree $\nu$ given by $f=\left(F_{0}:F_{1}:...:F_{n-1}\right)$  where the $F_{js}$, are homogeneous polynomials without common factors, with the same degree. 
\end{definition}

 Let us note that the indeterminacy locus $I(f)$ is the intersection of the hypersurfaces $\Pi_n(F_{i}=0)$ and $\Pi_n(F_{j}=0)$ for $i\neq j$.

\begin{definition}\label{generic} We say that $f  \in RM\left(n,n-1,\nu\right)$ is $generic$ if for all  \break$p \in$ $\tilde{f}^{-1}\left(0\right)\backslash\left\{0\right\}$ we have $dF_{0}\left(p\right)\wedge dF_{1}\left(p\right)\wedge...\wedge dF_{n-1}\left(p\right) \neq 0.$  
\end{definition}

This is equivalent to saying that $f  \in RM\left(n,n-1,\nu\right)$ is $generic$ if $I(f)$ is the transverse intersection of the $n$ hypersurfaces $\Pi_n(F_{i}=0)$ for $i=0,...,n-1.$  Moreover if $f$ is generic and $deg(f)=\nu$, then by Bezout's theorem $I\left(f\right)$ consists of $\nu^{n}$ distinct points. 

Now let $V(f)=\mathbb P^n\backslash I(f)$, $P(f)$ the set of critical points of $f$ in $V(f)$ and $C(f)=f(P(f))$ the set of the critical values of $f$. If $f$ is generic, then $\overline{P(f)}\cap I(f)=\emptyset$, so that  $\overline{P(f)}=P(f)\subset V(f)$ (where $\overline A$ denotes the closure of $A\subset \mathbb P^n$ in the usual topology). 
Since $P(f)=\{p\in V(f);rank(df(p)\leq (n-2)\}$, it follows from Sard's theorem that $C(f)=f(P(f))$ is a subset of Lebesgue's measure $0$ in $\mathbb P^{n-1}$, in fact $C(f)$ is an algebraic curve.

The set of generic maps will be denoted by $Gen\left(n,n-1,\nu\right)$. We state the following result, whose proof is standard in algebraic geometry:
\begin{proposition} $Gen\left(n,n-1,\nu\right)$ is a Zariski dense subset of $RM\left(n,n-1,\nu\right)$.
\end{proposition}

\section{Generic pull-back components - Generic conditions}\label{section4}
\begin{definition} Let $f \in Gen\left(n,n-1,\nu\right)$. We say that  
$\mathcal G \in A(d)$ is in generic position with respect to $f$ if $Sing\left(\mathcal G\right)\cap C(f)=\emptyset$ .
\end{definition}
\noindent In this case we say that $\left(f,\mathcal{G}\right)$ is a generic pair. In particular, when we fix a map $f\in Gen\left(n,n-1,\nu\right)$ the set $\mathcal{A}=\left\{\mathcal{G} \in A\left(d\right) | Sing\left(\mathcal G\right) \cap C(f)= \emptyset \right\}$ is an open and dense subset in $A(d)$ \cite{lnsc}, since $C(f)$ {is an algebraic curve in}  ${\mathbb P^{n-1}}.$ The set $U_{1}:=\{ (f,\mathcal G) \in Gen\left(n,n-1,\nu\right)\times A\left(d\right)| Sing\left(\mathcal G\right)\cap C(f)= \emptyset \}$ is an open and dense subset of $Gen\left(n,n-1,\nu\right)\times A\left(d\right)$. Hence the set $\mathcal W:=\left\{\tilde\eta_{[f,\mathcal G]}| \left(f,\mathcal G\right)\in U_{1}\right\}$ is an open and dense subset of $PB\left(\Theta;2,n\right)$.

Consider the set of foliations $\mathbb{F}{\rm{ol}}\left(d;1,n-1\right)$, $d\geq 2,$ and the following map:
\begin{eqnarray*}
\Phi:RM\left(n,n-1,\nu\right) \times\mathbb{F}{\rm{ol}}\left(d;1,n-1\right) &\to&\mathbb{F}{\rm{ol}}\left(\Theta;2,n\right)\\
\left(f,\mathcal G\right) &\to& f ^{\ast}\left(\mathcal G\right)=\Phi\left(f,\mathcal G\right).
\end{eqnarray*}

The image of $\Phi$ can be written as:  $$\left[
 (-1)^{i+k+1}\sum_{i,k} F_{k}(P_{i}\circ\tilde{f})\space dF_0\wedge...\wedge \widehat{dF_i}\wedge...\wedge\widehat{dF_k}\wedge...\wedge dF_{n-1}\right],$$ $i,k \in \{0,...,n-1\}$. Recall that $\Phi\left(f,\mathcal G\right)=\tilde\eta_{[f,\mathcal G]}$. More precisely, let $PB(\Theta,2,n)$ be the closure in $\mathbb{F}{\rm{ol}}\left(\Theta;2,n\right)$ of the set of foliations $\mathcal{F}$ of the form $f^*\left(\mathcal{G}\right)$, where $f\in RM\left(n,n-1,\nu\right)$ and $\mathcal G \in \mathbb{F}{\rm{ol}}\left(d;1,n-1\right).$ Since $RM\left(n,n-1,\nu\right)$ and $\mathbb{F}{\rm{ol}}\left(d;1,n-1\right)$ are irreducible algebraic sets and the map $\left(f,\mathcal{G}\right) \to f^*\left(\mathcal{G}\right) \in\mathbb{F}{\rm{ol}}\left(\Theta;2,n\right)$ is an algebraic parametrization of $PB(\Theta,2,n)$, we have that $PB(\Theta,2,n)$ is an irreducible algebraic subset of $\mathbb{F}{\rm{ol}}\left(\Theta;2,n\right)$. Moreover, the set of generic pull-back foliations $\left\{\mathcal{F}; \mathcal{F} = f^*(\mathcal{G}),\ \text{where} \ \left(f,\mathcal{G}\right)\right.$ is a generic pair$\left.\right\}$ is an open (not Zariski) and dense subset of $PB(\Theta,2,n)$ for $\nu \geq 2, d\geq2$. 
\begin{remark} We observe that if $\nu=1$ the theorem is also true and, in this case, we re-obtain the result  \cite[ Cor. 1 p.7]{ln3} and  \cite[Cor. 5.1 p. 426]{cupe} for the case of bi-dimensional foliations. \end{remark}
\begin{remark} To visualize that the degree of a generic pull-back foliation is indeed $\Theta(\nu,d,n)=[(d+n-1)\nu-3]$,  do the pull-back of a generic map of the Joanoulou's  foliation on $\mathbb P^{n-1}$ to obtain that the degree of this generic element coincides with this number.
 \end{remark}
\section{Description of generic pull-back foliations on $\mathbb P^n$}
\subsection{The Kupka set of $\mathcal{F} = f^*(\mathcal{G})$}\label{section5.1}

Let $q_i$ be a singularity of $\mathcal{G}$ and $V_{q_i}=\overline{f^{-1}(q_i)}$.  If $(f,\mathcal G)$ is a generic pair then $V_{q_i}\backslash I(f)$ is contained in the Kupka set  of $\mathcal F$. 

Fix $p\in V_{q_i}\backslash I(f)$. Since $f$ is a submersion at $p$, there exist local analytic coordinate systems $(U,y,t), y:U\to\mathbb C^{n-1}$, $t:U\to\mathbb C$, and $(V,u), u:V\to\mathbb C^{n-1}$, at $p$ and $q_i=f(p)$ respectively, such that  $f(y_1,y_2,...,y_{n-1},t)=(y_1,y_2,...,y_{n-1})$, $u(q)=0$. Suppose that $\mathcal {G}$ is represented by the vector field $Q=\sum_{j=1}^{n-1}Q_{j}(u)\frac{\partial}{\partial u_j}$ in a neighborhood of $q_i$. Then $\mathcal F$ is represented by $Y=\sum_{j=1}^{n-1}Y_{j}(y)\frac{\partial}{\partial y_j}$. It follows that in $U$, the foliation $\mathcal F$ is equivalent to the product of two foliations of dimension one: the singular foliation induced by the vector field $Y$ in $(\mathbb C^{n-1},0)$ and the regular foliation of dimension one given by the fibers of the first projection $F(y,t)=y$. Note that the curve $\gamma:=\{(y,t)|y=0\}$ is contained in the singular set of $\mathcal F$. Moreover, the condition $\mathcal G\in A(d)$ implies that $Div(Y(p))\neq0$. This is the Kupka-Reeb phenomenon, and we have that $p$ is in the Kupka set of $\mathcal F$ (see \cite{ln3}). It is known that this local product structure is stable under deformations of $\mathcal F$. Therefore if $p$ is as before it belongs to the Kupka-set of $\mathcal {F}$. It is known that this local product structure is stable under small perturbations of $\mathcal F$ \cite{kupka},\cite{medeiros}. 
 
\begin{remark} Note that, $\mathcal F$ has other singularities which are contained in $f^{-1}(C(f))$. We remark that $\mathcal F$ has local holomorphic first integral in a neighborhood of each singularity of this type. In fact, this is the obstruction to try to apply the results contained in \cite{cupe} to prove theorem A since these pull-back foliations do not have totally decomposable tangent sheaf. 

Since $\mathcal G$ has degree $d$ and all of its singularities are non degenerate it has $N=\frac{d^{n}-1}{d-1}$ singularities, say, $q_1,...,q_N$. We will denote the curves $\overline{f^{-1}(q_1)},\dots,\overline{f^{-1}(q_N)}$ by $V_{q_1}, \dots,V_{q_N}$ respectively. We have the following:

\begin{proposition} For each $\{j=1,\dots,N\}$,  $V_{q_j}$ is a complete intersection of $(n-1)$ transversal algebraic hypersurfaces. Furthermore,  $V_{q_j}\backslash I(f)$ is contained in the Kupka set of $\mathcal F=f^\ast\mathcal G$.
\end{proposition}   

\end{remark}
\subsection{Generalized Kupka singularities for $2$-dimensional foliations}\label{section5.2}
 In this section we will recall the generalized Kupka singularities of an integrable holomorphic $(n-2)$-form, for more detail we refer the reader to  \cite{ln3}.  They appear in the indeterminacy set of $f$ and play a central role in great part of the proof of the main theorem. 
 Let  $\Omega$ be an holomorphic  integrable $(n-2)$-form defined in a neighborhood of $p \in {\mathbb{C}^n}$. In particular, since $d\Omega$ is a $(n-1)$-form, there exists a holomorphic vector field $\mathcal Z$ defined in a neighborhood of $p$ such that: $d\Omega=i_{\mathcal Z}dw_0\wedge\dots\wedge dw_{n-1}$. 
\begin{definition}
We say that $p$ is a  singularity of generalized Kupka type of $\Omega$ if $\mathcal Z(p)=0$ and $p$ is an isolated zero of $\mathcal Z$.   
\end{definition}
\begin{definition} We say that $p$ is a nilpotent generalized singularity, for short n.g.k singularity, if the linear part of $\mathcal Z$, $D\mathcal Z(p)$ is nilpotent. 
\end{definition}
\noindent This definition is justified by the following result (that can be found in \cite{ln3}).
 \begin{theorem} \label{teo5.3}
Assume that $0\in\mathbb C^{n}$ is a n.g.k singularity of $\Omega.$  Then there exists two holomorphic vector fields $S$ and $\mathcal Z$ and a holomorphic coordinate system $x=(x_0,...,x_{n-1})$ around $0\in\mathbb C^{n}$ where $\Omega$ has polynomial coefficients. More precisely, there exists two polynomial vector fields $X$ and $Y$ in $\mathbb C^{n}$ such that:
\begin{enumerate}
\item[(a)] $Y=S+N$, where $S=\sum_{j=0}^{n-1}k_{j}w_j\frac{\partial}{\partial w_j}$
is linear semi-simple with eigenvalues $k_0,...,k_{n-1}\in\mathbb N$, $DN(0)$ is linear nilpotent and $[S,N]=0$.
  \item[(b)] $[N,X]=0$ and $[S,X]=kX$, where $k\in \mathbb N$. In other words, $X$ is quasi-homogeneous with respect to $S$ with weight $k$.
\item[(c)]  In this coordinate system we have $\Omega=i_Yi_Xdx_0\wedge\dots\wedge dx_{n-1}$ and $L_Y(\Omega)=(k+tr(S)\Omega)$.
\end{enumerate} In particular, the foliation given by $\Omega=0$ can be defined by a local action of the affine group.
\end{theorem}
 \begin{definition}
 In the situation of the theorem \ref{teo5.3}, $S=\sum_{j=0}^{n-1}k_{j}x_j\frac{\partial}{\partial x_j}$ and $L_S(X)=kX$, we say that the n.g.K is of type $(k_0,...,k_{n-1};k)$. 
  \end{definition}

 \begin{remark} \label{remark sobre o N ser nulo}
We would like to observe that in many cases it can be proved that the vector field $N$ of the statement of theorem 2 vanishes. In order to discuss this assertion it is convenient to introduce some objects. Given two germs of vector fields Z and W set $L_Z(W):=[Z,W]$. Recall that $\Sigma(S,\ell) = \{Z \in X |L_S(Z) = \ell Z\}$. Let $X$ and $Y =S+N$ be as in theorem \ref{teo5.3}. Observe that:
 \begin{itemize}

\item Jacobi's identity implies that if $W \in \Sigma(S, k)$ and $Z \in \Sigma(S, \ell)$ then $[W, Z] \in \Sigma(S, k + \ell).$
\item For all $k \in \mathbb Z$ we have $dim_{\mathbb C}(\Sigma(S, k)) < \infty$ (because $k_0,...,k_{n-1} \in \mathbb N)$.
\item $N \in \Sigma(S, 0), X \in \Sigma(S, \ell)$ and $L_X(N) = 0$, so that $N\in ker(L^0_X)$, where $L^0_X := LX : \Sigma(S, 0) \to \Sigma(S, \ell)$. In particular, the vector field $N \in\Sigma(S, 0)$
of theorem \ref{teo5.3} necessarily vanishes  $\iff$ $ker(L^0_X)=\{0\}$.
 \end{itemize}
 In \cite{ln3},  $\S$ 3.2 it is shown that under a non-resonance condition, which depends only on $X$, then $ker(L^0_X)=\{0\}$. Let us mention some correlated facts.\begin{enumerate}
\item[(I)]
  If $S$ has no resonances of the type $<\sigma,k>-k_j=0$, where $<\sigma,k> = \sum_j \sigma_j.k_j$,
$k = (k_0,...,k_{n-1})$ and $\sigma = (\sigma_0,...,\sigma_{n-1}) \in \mathbb Z^n_{\geq0}$, then $ker(L^0_X)=\{0\}$.
\item[(II)] When $n = 3$ and $X$ has an isolated singularity at $0 \in \mathbb C^3$ then $ker(L^0_X)=\{0\}$(c.f \cite{ln0}).
 
 \item[(III)] When $N \not\equiv0$ and $cod_\mathbb C(sing(N))=1$, or $sing(N)$ has an irreducible component of dimension one then it can be proved that $X$ cannot have an isolated singularity at $0\in \mathbb C^n$.
 \end{enumerate}
 
In fact, we think that whenever $X$ has an isolated singularity at $0\in \mathbb C^n$ and $\nabla X=0$ then $ker(L^0_X)=\{0\}$.
 \end{remark}
The next result is about the nature of the set $\mathcal K(S,\ell):=\{X \in\Sigma(S, \ell)| ker(L^0_X)=\{0\}$ and $\nabla X = 0\}$.
 
 \begin{proposition}
\label{prop5.7} If $\mathcal K(S,\ell)\neq\emptyset$ then $\mathcal K(S,\ell)$ is a Zariski open and dense subset of $\mathcal E (S,\ell)$. In particular, if there exists $X \in \mathcal E (S,\ell)$ satisfying the non-resonance condition mentioned in remark \ref{remark sobre o N ser nulo} then $\mathcal K(S,\ell)$ is a Zariski open and dense in $\mathcal E (S,\ell)$.
Proposition \ref{prop5.7} is a straightforward consequence of the following facts:
 \begin{enumerate}
\item[(A)]The set of linear maps $\mathcal L(\Sigma(S, 0), \Sigma(S, \ell))$ is finite dimensional vector space. Moreover, the subspace $\mathcal N I := \{T \in \mathcal L(\Sigma(S, 0), \Sigma(S, \ell))| $ $T$ is not injective$\}$ is an algebraic subset of $\mathcal L(\Sigma(S, 0), \Sigma(S, \ell))$.
\item[(B)] The map $L:\mathcal E (S,\ell) \to \mathcal L(\Sigma(S, 0), \Sigma(S, \ell))$ defined by $L(X) = L^0_X$ is linear. As a consequence, the set $L^{-1}(\mathcal NI)$ is an algebraic subset of $\mathcal E (S,\ell)$.
\item[(C)] $\mathcal K(S,\ell)=\mathcal E (S,\ell) \backslash L^{-1}(\mathcal NI)$.
 \end{enumerate} \end{proposition}

We leave the details to the reader.
\begin{remark} 
  In the case of the radial vector field, $R=\sum_{i=0}^{n-1}x_i\frac{\partial}{\partial x_i}$ , we have $\mathcal K(R,d-1)\neq\emptyset$  for all $d \geq2$. In fact, it is proved in \cite{ln3},  $\S$ 3.2  that $J_d \in \mathcal(R,d-1)$, where $J_d$ is the generalized Jouanolou's vector field.  
 \end{remark}

 Consider a holomorphic family of $(n-2)$-forms, $(\Omega_t)_{t\in U}$, defined on a polydisc $Q$ of $\mathbb C^n$, where the space of parameters $U$ is an open set of $\mathbb C^k$ with $0\in U$. Let us assume that:
 \begin{itemize}
 \item For each $t\in U$ the form $\Omega_t$ defines a $2$-dimensional foliation $\mathcal F_t$ on $Q$. Let $(\mathcal Z_t)_{t\in U}$ be the family of holomorphic vector fields on $Q$ such that $d\Omega_t=i_{\mathcal Z_t}dx_0\wedge\dots\wedge x_{n-1}$.  
 \item $\mathcal F_0$ has a n.g.K singularity at $0\in Q$.
 \end{itemize}
 We can now state the stability result, whose proof can be found in \cite{ln3}:
 \begin{theorem}\label{stability} In the above situation there exists a neighborhood $0\in V\subset U$, a polydisc $0\in P \subset Q$, and a holomorphic map $\mathcal P:V\to P\subset \mathbb C^n$ such that $\mathcal P(0)=0$ and for any $t\in V$ then $\mathcal P(t)$ is the unique singularity of $\mathcal F_{t}$ in $P$. Moreover, $P(t)$ is the same type as $\mathcal P(0)$ in the sense that: If $0$ is a n.g.K singularity of type $(k_0,\dots,k_{n-1},k)$ of $\mathcal F_{0}$ then  $\mathcal P(t)$ is a n.g.K singularity of type $(k_0,\dots,k_{n-1},k)$ of $\mathcal F_{t}$, $\forall t\in V$.
 \end{theorem}

 Let us now describe  $\mathcal{F} = f^*(\mathcal{G})$ in a neighborhood of a point $p \in I(f).$
 \begin{proposition} \label{lema5.6}If $p\in I(f)$ then $p$ is a n.g.K singularity of $\Omega$ of type $(1,\dots,1,n)$.   
 \end{proposition} 

\begin{proof}\noindent It is easy to show that there exists a local chart $(U,x=(x_0,...,x_{n-1})\in \mathbb C^n$  around $p$ such that the lifting $\tilde f$ of $f$ is of the form $\tilde f|_{U}=(x_0,...,x_{n-1}):U \to \mathbb{C}^{n}$. In particular $\mathcal {F}|_{U(p)}$ is represented by the homogeneous $(n-2)$-form $$\Omega=(-1)^{i+k+1}\sum_{i,k} x_{k}P_{i}dx_0\wedge...\wedge \widehat{dx_i}\wedge...\wedge\widehat{dx_k}\wedge...\wedge dx_{n-1}.$$ 
Observe that $L_R\Omega=(d+n)\Omega$, $\mathcal X=\sum_{i}P_i\frac{\partial}{\partial x_i}$, $\mathcal Z=(d+n)\mathcal X$, $[R,\mathcal X]=d\mathcal X$. Since we are considering $\Omega\in \mathcal A$ we  have that $Y=R$, $N=0$ and hence we conclude that $p$ is a n.g.K singularity of $\Omega$ of type $(1,\dots,1,n)$.
In particular the vector field $S$ as in the Theorem \ref{teo5.3} is the radial vector field. \end{proof} 
It follows from theorem. \ref{stability} that this singularities are stable under deformations. 
Propositon \ref{lema5.6} says that the germ $f^\ast \mathcal G,p$ of $f^\ast \mathcal G$ at $p\in I(f)$ is equivalent to the germ of $\Pi_{n-1}^\ast(\mathcal G),0$, where $\Pi_{n-1}: \mathbb C^n\backslash \{\ 0\}\to \mathbb P^{n-1}$ is the canonical projection. Therefore in this case we can see all the foliations $\mathcal G$ in a neighborhood of each $p\in I(f)$.
\subsubsection{Deformations of the singular set of ${\mathcal F}_0= f_{0}^\ast (\mathcal G_{0})$}

 We have constructed an open and dense subset $\mathcal W$ inside {$PB(\Theta,2,n)$} containing the generic pull-back foliations.  We will show that for any rational foliation $\mathcal{F}_{0} \in \mathcal W$ and any germ of a holomorphic family of foliations $(\mathcal{F}_{t})_{t \in (\mathbb{C},0)}$ such that $\mathcal{F}_{0}=\mathcal{F}_{t=0}$ we have $\mathcal{F}_{t} \in PB(\Theta,2,n)$ for all $t \in (\mathbb{C},0).$

Using the theorem. \ref{stability} with $V=(\mathbb C,0)$, it follows that for each $p_j\in I(f_0)$ there exists a deformation $p_{j}(t)$ of $p_{j}$ and a deformation of $\mathcal F_{t,p_j(t)}$ of $\mathcal F_{p_{j}}$ such that $p_{j}(t)$  is a n.g.K singularity of $\mathcal F_{t,p_j(t)}:=\Omega_{p_j(t)}$ of type $(1,\dots,1,n)$ and $(\mathcal F_{t,p_j(t)})_{t \in (\mathbb{C},0)}$ defines a holomorphic family of foliations in $\mathbb P^{n-1}$.
   We will denote by $I(t)=\{p_{1}(t),..., p_{j}(t),...,p_{\nu^{n}}(t)\}$.  
   \begin{remark}Since $I(t)$ is not connected we can not guarantee a priori that $\mathcal F_{t,p_i(t)}=\mathcal F_{t,p_j(t)}$, if $i \neq j.$
   \end{remark}

\begin{lemma} 
There exist $\epsilon >0$ and smooth isotopies $\phi_{q_i}:D_{\epsilon}\times V_{q_i} \to \mathbb P^n,  q_i \in Sing(\mathcal G_{0})$, such that $V_{q_i}(t)=\phi_{i}(\{t\}\times V_{q_i})$ satisfies:
\begin{enumerate}
\item[(a)] $V_{q_i}(t)$ is an algebraic subvariety of dimension $1$ of $\mathbb P^n$ and  $V_{q_i}(0) = V_{q_i}$ for all $ q_i\in Sing(\mathcal G_{0})$ and for all $t \in D_{\epsilon}.$
\item[(b)] $I (t) \subset V_{q_i}(t)$ for all $q_i \in Sing(\mathcal G_{0})$ and for all $t \in D_{\epsilon} $. Moreover, if $q_i \neq q_j$, and $q_i, q_j \in Sing(\mathcal G_{0})$,  we have $V_{q_i}(t)  \cap V_{q_j}(t) = I(t)$ for all $t \in D_{\epsilon}$ and the intersection is transversal.
\item[(c)]   $V_{q_i}(t) \backslash I(t)$ is contained in the Kupka-set of $\mathcal{F}_t$ for all $q_i \in Sing(\mathcal G_{0})$ and for all $t \in D_{\epsilon}.$ In particular, the transversal type of $\mathcal F_{t}$ is constant along $V_{q_i}(t) \backslash I(t)$. 
\end{enumerate}
\end{lemma}
\begin{proof}  See \cite[lema 2.3.3, p.83]{ln}.
 \end{proof}

\subsection{End of the proof of Theorem \ref{teob}}\label{subsection5.4} 

We divide the end of the proof of Theorem \ref{teob} in two parts. In the first part we construct a family  of rational maps $f_{t}:{\mathbb{P}^n \DashedArrow[->,densely dashed    ]   \mathbb{P}^{n-1}}$, $f_{t} \in Gen\left(n,n-1,\nu\right)$, such that $(f_{t})_{t \in D_{\epsilon}}$ is a deformation of $f_{0}$ and the subvarieties $V_{q_i}(t)$ are fibers of $f_t$ for all $t$. In the second part we show that there exists a family of foliations $(\mathcal {G}_{t})_{t\in D_{\epsilon}}, \mathcal {G}_{t} \in  \mathcal A$ (see Section \ref{section4}) such that  $\mathcal{F}_{t}= f_{t}^{ \ast}(\mathcal{G}_{t})$ for all $t\in D_{\epsilon}$. 

\subsubsection{Part 1}\label{part1} 

 Once $d=deg(\mathcal G_{0})\geq 2$, the number of singularities of $\mathcal G_{0}$ is $N=\frac{d^{n}-1}{d-1}>{n}$, so we can suppose that the singularities of $\mathcal G_0$ are $q_{1}=[0:0:...:1],...,q_{n}=[1:0:...:0],\dots,q_{N}$. 

  \begin{proposition}\label{recupmapas}
Let $(\mathcal{F}_{t})_{t \in D_{\epsilon}}$ be a deformation of $\mathcal F_{0}= f_{0}^*(\mathcal G_{0})$, where $(f_{0}, \mathcal G_{0})$ is a generic pair, with $\mathcal G_{0} \in \mathcal A$, $f_0 \in Gen\left(n,n-1,\nu\right)$ and $deg(f_{0})=\nu\geq 2$. Then there exists a deformation $({f}_{t})_{t \in D_{\epsilon}}$ of $f_{0}$ in $Gen\left(n,n-1,\nu\right)$ such that:
\begin{enumerate}
\item[(i)]$V_{q_i}(t)$ are fibers of $({f}_{t})_{t \in D_{\epsilon'}}$.
\item[(ii)] $I(t)=I(f_{t}), \forall {t \in D_{\epsilon'}}$.
    \end{enumerate}
  \end{proposition}  
\begin{proof} Let $\tilde f_{0}=(F_0,...,F_{n-1}): \mathbb C^{n+1} \to\mathbb C^{n}$ be the homogeneous expression of $f_{0}$. Then $V_{q_1}$, $V_{q_2}$ ,..., and $V_{q_n}$ appear as the complete intersections $V_{q_i}=(F_{0}=F_{1}=...=\widehat{F_{i-1}}=...=0)$. The remaining fibers, $V_{q_i}$ for $i>n$  are obtained in the same way.  With this convention we have that $I(f_0)=V_{q_i}\cap V_{q_j}$ if $i\neq j$. It follows from \cite{Ser0} (see section 4.6 pp 235-236) that each $V_{q_i}(t)$ is also a smooth complete intersection generated by polynomials of the same degree. However we can not assure that the set of polynomials which define each $V_{q_i}(t)$ have a correlation among them. In the next lines we will show this fact. For this let us work firstly with two curves. For instance, let us take $V_{q_1}(t)$ and $V_{q_2}(t)$ which are deformations of $V_{q_1}$ and $V_{q_2}$ respectively.  We will see that this two curves are enough to construct the family of deformations $({f}_{t})_{t \in D_{\epsilon'}}$.  After that we will show that the remaining curves $V_{q_i}(t)$ are also fibers of $({f}_{t})_{t \in D_{\epsilon'}}$. We can write $V_{q_1}(t)=(F_{1}(t)=F_{2}(t)=...F_{n-1}(t)=0)$, and  $V_{q_2}(t)=(\tilde F_{0}(t)={F_{1}(t)}=\tilde F_{2}(t)...=...\tilde F_{n-1}(t)=0)$ where $(F_{i}(t))_{t \in D_{\epsilon'}}$  and  $(\tilde{F}_{i}(t))_{t \in D_{\epsilon'}}$ are deformations of $F_{i}$ and $ D_{\epsilon'}$ is a possibly smaller neighborhood of $0$. Observe first that since the $F_{is}(t)$ and $\tilde F_{is}(t)$  are near $F_{is}$, they meet as a complete intersection at:
$$I(f_t):=(F_{0}(t)=0)\cap V_1(t)$$
On the other hand we also have 
$$I(t)=V_{q_1}(t)\cap V_{q_2}(t)=V_{q_1}(t)\cap[(F_{0}(t)=0)\bigcap \{\tilde F_{2}(t)=\dots\tilde F_{n-1}(t)=0\}].$$
Let us write $\{S(t)=0\}=\{\tilde F_{2}(t)=\dots\tilde F_{n-1}(t)=0\}$. Hence $I(f_t) \cap \{S(t)=0\}=V_{q_1}(t)\cap V_{q_2}(t)=I(t)$, which implies that $I(t)\subset I(f_t).$  Since $I(f_t)$ and $I(t)$ have $\nu^{n}$ points, we have that  $I(t)=I(f_t)$ for all $t \in D_{\epsilon'}$.  In particular, we obtain that $I(t)\subset \{S(t)=0\}$.  
We will use the following version of Noether's Normalization Theorem (see \cite{ln} p 86):

\begin{lemma} (Noether's Theorem) Let $G_{0},...,G_{k}  \in \mathbb {C}[z_{1},...,z_{m}]$ be homogeneous polynomials where $0 \leq k \leq m$ and $m \geq 2$, and $X=(G_{0}=...=G_{k}=0)$. Suppose that the set $Y:=\{p \in X | dG_{0}(p) \wedge...\wedge dG_{k}(p)=0\}$ is either $0$ or $\emptyset$. If $G  \in \mathbb {C}[z_{1},...,z_{m}]$ satisfies $G|_{X} \equiv 0$, then  $G$  $\in$  $<G_{0},...,G_{k}>$. 
\end{lemma} 
\noindent Take $k=n-1$, $G_{0}=F_{0}(t)$, $G_{1}=F_{1}(t)\dots G_{n-1}=F_{n-1}(t)$. Using Noether's Theorem with $Y=0$ and the fact that all polynomials involved are homogeneous of the same degree, we have $\tilde F_{i}(t)$ $\in$ $<F_{0}(t),F_{1}(t),\dots,F_{n-1}(t)>$. More precisely we conclude that each $\tilde F_{i}(t)=\sum _{j=0}^{n-1}g_{i,j}(t)F_{j}(t)$, $g_{i,j}(t)\in\mathbb C$ and when $t=0$ for each $i$, $\tilde F_{i}(0)=F_{i}(0)=F_{i}$. On the other hand, if $V_{q_k}(t)$ is the deformation of another $V_{q_k}$,  then $V_{q_k}(t)$ is also a complete intersection , say, $V_{q_k}(t)=(P_1^{k}(t)=\dots=P_{n-1}^{k})=0$ where each $P_i^{k}(t)$, for $ i=1,\dots,{n-1}$ is a homogeneous polynomial of degree $\nu$.
Since $I(t)\subset(P_i^{k}(t)=0)$ for  $ i=1,\dots,{n-1}$, we have that each  $P_i^{j}(t)$ is a linear combination of the $F_{is}(t)$, that is, $P_i^{k}(t)$ $\in$ $<F_{0}(t),F_{1}(t),\dots,F_{n-1}(t)>$. This implies that $V_{q_k}(t)$ is also a fiber of $f_t$, as the reader can check, say $V_{q_k}(t)=f_t^{-1}(q_k(t))$. Since $q_{k}(t)=f_{t}(V_{q_k}(t))$ and $f_t$ and $V_{q_k}(t)$ are deformations of $f_0$ and $V_{q_k}$ we get that  $q_{k}(t)$ is a deformation of $q_{k}$, so that for small $t$, $q_{k}(t)$ is a regular value of  $f_t$.
\end{proof}
\subsubsection{Part 2} Let us now define a family of foliations $(\mathcal {G}_{t})_{t\in D_{\epsilon}}, \mathcal {G}_{t} \in  \mathcal A$ (see Section \ref{section4}) such that  $\mathcal{F}_{t}= f_{t}^{ \ast}(\mathcal{G}_{t})$ for all $t\in D_{\epsilon}$.   Let $M(t)$ be the family of ``rational varieties" obtained from $\mathbb P^n$ by blowing-up at the $\nu^{n}$ points $p_{1}(t),..., p_{j}(t),...,p_{\nu^{n}}(t)$ corresponding to $I(t)$ of $\mathcal F_{t}$; and denote by $$\pi(t):M(t) \to \mathbb P^n$$ the blowing-up map. The exceptional divisor of $\pi(t)$ consists of $\nu^{n}$ submanifolds $E_{j}(t)=\pi(t)^{-1}(p_{j}(t)),$ $1\leq j \leq \nu^{n}$, which are projective spaces $\mathbb P^{n-1}$. 
More precisely, if we blow-up $\mathcal{F}_{t}$ at the point $p_{j}(t)$, then the restriction of the strict transform $\pi^{\ast}\mathcal{F}_{t}$ to the exceptional divisor $E_{j}(t)=\mathbb P^{n-1}$ is up to a linear automorphism of  $\mathbb P^{n-1}$, the homogeneous $(n-2)$-form that defines $\mathcal{F}_{t}$ at the point $p_{j}(t)$.  With this process we produce a family of bidimensional holomorphic foliations in $\mathcal A$. This family is the ``holomorphic path'' of candidates to be a deformation of $\mathcal G_{0}$. In fact, since $\mathcal A$ is an open set we can suppose that this family is inside $\mathcal A$. 
We fix the exceptional divisor $E_1(t)$ to work with and we denote by $\mathcal{G}_t$ the restriction of $\pi^*\mathcal{F}_t$ to $E_1(t)$.  Consider the family of mappings ${f}_{t}:\mathbb P^{n} \DashedArrow[->,densely dashed    ]\mathbb P^{n-1}$, ${t \in D_{\epsilon'}}$ defined in Proposition \ref{recupmapas}. We will consider the family $(f_{t})_{t\in D_{\epsilon}}$ as a family of rational maps  $f_{t}:\mathbb P^n  \DashedArrow[->,densely dashed    ] E_{1}(t)$; we decrease $\epsilon$ if necessary. We we would like to observe that the mapping $f_{t}\circ\pi(t):M(t)\backslash \cup_{j} E_{j}(t)\to\mathbb P^{n-1}$ extends as holomorphic mapping  $\hat{f}_{t}:M(t) \to \mathbb P^{n-1}$ if $|t|<\epsilon$. This follows from the fact that $dF_{0}(t)\left(p_{j}(t)\right)\wedge dF_{1}(t)\left(p_{j}(t)\right)\wedge...\wedge dF_{n-1}(t)\left(p_{j}(t)\right) \neq 0$, $1\leq j\leq\nu^n$, if $|t|<\epsilon$. The mapping  $f_t$ can be interpreted as follows. 
Each fiber of $f_t$ meets $p_{j}(t)$ once, which implies that each fiber of $\hat{f}_{t}$ cuts $E_{1}(t)$ only one time.  Since $M(t) \backslash \cup_{j} E_{j}(t)$ is biholomorphic to $\mathbb P^n \backslash I(t)$, after identifying $E_{1}(t)$ with $\mathbb P^{n-1}$, we can imagine that if $q \in M(t) \backslash \cup_{j} E_{j}(t)$ then $\hat{f}_{t}(q)$ is the intersection point of the fiber $\hat{f}_{t}^{-1}(\hat{f}_{t}(q))$ with $E_{1}(t)$. We obtain a mapping $$\hat{f}_{t}:M(t) \to \mathbb P^{n-1}.$$ With all these ingredients we can define the foliation $\tilde{\mathcal{F}}_{t}={f}_{t}^\ast(\mathcal G_{t}) \in PB(\Theta,2,n)$. This foliation is a deformation of $\mathcal{F}_0$.  Based on the previous discussion let us denote ${\mathcal{F}_{1}}({t})= \pi(t)^{\ast} ({\mathcal{F}}_{t})$ and $\hat{\mathcal{F}_{1}}({t})= \pi(t)^{\ast} (\tilde{\mathcal{F}}_{t})$. \begin{lemma}If ${\mathcal{F}_{1}}({t})$ and $\hat{\mathcal{F}_{1}}({t})$ are the foliations defined previously, we have that $${\mathcal{F}_{1}}({t})|_{{E_{1}(t)}\simeq \mathbb P^{n-1}}={\hat {\mathcal G}_{t}}=\hat{\mathcal{F}_{1}}({t})|_{{E_{1}(t)}\simeq \mathbb P^{n-1}}$$ where ${\hat {\mathcal G}_{t}}$ is the foliation induced on ${E_{1}(t)}\simeq \mathbb P^{n-1}$ by the homogeneous $(n-2)$-form $\Omega_{p_1(t)}$. 
\end{lemma}

\begin{proof} In a neighborhood of $p_1(t) \in I(t)$, ${\mathcal{F}}_{t}$ is represented by the homogeneous $(n-2)$-form $\Omega_{p_1(t)}$. This $(n-2)$-form satisfies $i_{R(t)}\Omega_{p_1(t)}=0$ and therefore naturally defines a foliation on $\mathbb P^{n-1}$. This proves the first equality. The second equality follows from the geometrical interpretation of the mapping $\hat{f}_{t}:M(t) \to  \mathbb P^{n-1}$, since $\hat{\mathcal{F}_{1}}({t})=f_{1}(t)^{*}(\mathcal {G}_{t})$. \end{proof}

Let ${q}_{1}(t)$ be a singularity of $\mathcal G_{t}$. Since the map $t \to {q}_{1}(t) \in \mathbb P^{n-1}$ is holomorphic, there exists a holomorphic family of automorphisms of $\mathbb P^{n-1}$, $t \to H(t)$ such that  ${q}_{1}(t)=[0:\dots:1]$ $\in {E_{1}(t)}$ is kept fixed. Observe that such a singularity has $(n-1)$ non algebraic separatrices at this point. Fix a local analytic coordinate system $(U_t=u_0(t),\dots,u_{n-1}(t))$ at ${q}_{1}(t)$ such that the local separatrices are tangentes to  $(u_i(t)=0)$ for each $i$. Observe that the local smooth hypersurfaces along $\hat V_{q_1(t)}=\hat{f}_{t}^{-1}({q}_{1}(t))$ defined by $\hat U_i(t):=(u_i(t)\circ\hat{f}_{t}=0)$ are invariant for $\hat{\mathcal{F}_{1}}({t})$. Furthermore, they meet transversely along  $\hat V_{q_1(t)}$. On the other hand, $\hat V_{q_1(t)}$ is also contained in the Kupka set of ${\mathcal{F}_{1}}({t})$.  Therefore there are $(n-1)$ local smooth hypersurfaces $U_i(t):=(u_i(t)\circ\hat{f}_{t}=0)$ invariant for ${\mathcal{F}_{1}}({t})$ such that:
\begin{enumerate}
\item All the $U_i(t)$ meet transversely along $\hat V_{q_1(t)}$.
\item$U_i(t)\cap\pi(t)^{-1}(p_{1}(t))=(U_i(t)=0)=\hat U_i(t)\cap\pi(t)^{-1}(p_{1}(t))$ (because ${\mathcal{F}_{1}}({t})$ and $\hat{\mathcal{F}_{1}}({t})$) coincide on ${{E_{1}(t)}\simeq \mathbb P^{n-1}}$).
\item Each $U_i(t)$ is a  deformation of $U_i(0)=\hat U_i(0)$. 
\end{enumerate}

Choosing $i=0$ we shall prove that $U_0(t)=\hat U_0(t)$ for small $t$. For our analysis this will be sufficient to finish the proof of Theorem A.

\begin{lemma}\label{lemafund}$U_0(t)=\hat U_0(t)$ for small $t$.
\end{lemma}
\begin{proof} Let us consider the projection $\hat{f}_{t}:M(t) \to  \mathbb P^{n-1}$ on a neighborhood of the regular fibre  $\hat V_{q_1(t)}$, and fix local coordinates $(U_t=u_0(t),\dots,u_{n-1}(t))$ on $\mathbb P^{n-1}$ such that $U_1(t):=(u_1(t)\circ\hat{f}_{t}=0)$.  For small $\epsilon$, let $H_{\epsilon}=(u_1(t)\circ\hat{f}_{t}=\epsilon)$. Thus $\hat\Sigma_\epsilon=\hat U_0(t)\cap H_\epsilon$ are (vertical) compact curves, deformations of $\hat\Sigma_0=\hat V_{q_1(t)}$. Set $\Sigma_\epsilon=U_0(t)\cap\hat H_\epsilon$. The $\Sigma_\epsilon's$, as the $\hat\Sigma_\epsilon's$, are compact curves (for $t$ and $\epsilon$ small), since $U_0(t)$ and $\hat U_0(t)$ are both deformations of the same $U_0$. Thus for small $t$, $U_0(t)$ is close to $\hat U_0(t)$. It follows that $\hat{f}_{t}(\Sigma_\epsilon)$ is an analytic curve contained in a small neighborhood $\tilde U_t$ of ${q_1(t)}$, for small $\epsilon$. By the maximum principle, we must have that $\hat{f}_{t}(\Sigma_\epsilon)$ is a point, so that  $\hat{f}_{t}(U_0(t))=\hat{f}_{t}(\cup_{\epsilon}\Sigma_\epsilon)$ is a curve $C\subset \tilde U_t$, that is, $U_0(t)=\hat{f}_{t}^{-1}(C)$. But $U_0(t)$ and $\hat U_0(t)$ intersect the exceptional divisor ${{E_{1}(t)}\simeq \mathbb P^{n-1}}$ along the separatrix $(u_0(t)=0)$ of $\mathcal G_t$ through ${q_1(t)}$. This implies that $U_0(t)=\hat{f}_{t}^{-1}(C)=\hat{f}_{t}^{-1}(u_0(t)=0)=\hat U_0(t)$.
\end{proof}

\noindent We have proved that the foliations $\mathcal F_t$ and $\tilde{\mathcal F_t}$ have a common local leaf: the leaf that contains $\pi(t)\left(U_0(t)\backslash\hat V_{q_1(t)}\right)$ which is not algebraic. Let $D(t):=Tang(\mathcal{F}(t),\hat {\mathcal{F}}(t))$ be the set of tangencies between $\mathcal{F}(t)$ and $\hat {\mathcal{F}}(t)$. This set can be defined by $D(t)=\{ Z \in \mathbb {C}^{n+1}; \Omega(t) \wedge\hat {\Omega}(t)=0\},$ where  $\Omega(t)$ and  $\hat{\Omega}(t)$ define $\mathcal{F}(t)$ and $\hat{\mathcal{F}}(t)$, respectively. Hence it is an algebraic set. Since this set contains an immersed non-algebraic surface $U_0(t)$, we necessarily have that  $D(t)=\mathbb P^{n}.$ It follows that ${\mathcal{F}}_{t}=\tilde{\mathcal{F}}_{t}$. \\

\noindent Recall from Definition \ref{generic} the concept of a generic map. Let $f  \in RM(\left(n,n-1,\nu\right)$, $I(f)$ its indeterminacy locus and $\mathcal F$ a foliation by complex surfaces on $\mathbb P^n$, $n\geq4$. Consider the following properties:

\begin{center}
\begin{minipage}{10cm}
$\mathcal{P}_1:$ Any point $p_{j}\in I(f)$ $\mathcal F$ has the following local structure: there exists an analytic coordinate system $(U^{p_{j}},x^{p_{j}})$ around $p_{j}$ such that $x^{p_{j}}(p_{j})=0 \in ({\mathbb C^n,0})$ and $\mathcal F|_{(U^{p_{j}},x^{p_{j}})}$ can be represented by a homogeneous $(n-2)$-form $\Omega_{p_j}$ (as described in the Lemma \ref{lema5.6}) such that
\begin{enumerate} 
\item[(a)] $Sing(\mathcal Z_{p_j}) = {0}$, where $\mathcal Z_{p_j}$ is the rotational of $\Omega_{p_j}$.
\item[(b)] $0$ is a n.g.K singularity of the type $(1,\dots,1,n)$
\end{enumerate}

\end{minipage}
\end{center}

\begin{center}
\begin{minipage}{10cm}
$\mathcal{P}_2:$ There exists a fibre $f^{-1}(q)=V(q)$ such that $V(q)=f^{-1}(q)\backslash I(f)$ is contained in the Kupka-Set of $\mathcal F$.
\end{minipage}
\end{center}

\begin{center}
\begin{minipage}{10cm}
$\mathcal{P}_3:$ $V(q)$ has transversal type $Q$, where $Q$ is a germ of vector field on $({\mathbb C^{n-1},0})$ with at least a non algebraic separatrix and such that the Camacho-Sad index of $\mathcal G$ with respect to this separatrix is non-real. 
\end{minipage}
\end{center}

\begin{center}
\begin{minipage}{10cm}
$\mathcal{P}_4:$ $\mathcal F$ has no algebraic hypersurface.
\end{minipage}
\end{center}

Lemma \ref{lemafund} allows us to prove the following result:

\begin{main}\label{teoc}
In the conditions above, if properties $\mathcal{P}_1$, $\mathcal{P}_2$, $\mathcal{P}_3$ and $\mathcal{P}_4$ hold then $\mathcal F$ is a pull back foliation, $\mathcal {F}= f^{*}(\mathcal{G})$, where $\mathcal{G}$ is a $1$-dimensional foliation of degree $d\geq2$ on $\mathbb P^{n-1}$.  
\end{main} 

Note that the situation $n=3$ is proved in  \cite[Th. B p. 709]{clne}. So we can think this result as $n\geq4$-dimensional generalization of \cite{clne} for bi-dimensioal foliations in $\mathbb P^{n}$
%



\bibliographystyle{amsalpha}

\end{document}